\documentclass[12pt]{article}

\usepackage{amssymb,amsmath,euscript}

\newtheorem{proposition}{Proposition}[section]
\newtheorem{lemma}[proposition]{Lemma}
\newtheorem{theorem}[proposition]{Theorem}
\newtheorem{definition}[proposition]{Definition}
\newtheorem{corollary}[proposition]{Corollary}
\def\a{\alpha}

\def\De{\Delta}

\def\In{\mathcal{I}}
\def\ga{\gamma}

\def\La{\Lambda}
\def\ep{\varepsilon}
\def\si{\sigma}

\def\ka{\kappa}
\def\th{\theta}
\def\o{\omega}

\def\O{\Omega}
\def\car{{\rm card}}

\def\diam{{\rm diam}}

\def\Sp{{\rm supp}}

\def\jb{\overline{J}_0}
\def\jbu{\overline{J}_1}
\def\jbut{\overline{J}_2(\tau)}

\def\li{{\lfloor}}
\def\ri{{\rfloor}}

\def\F{{\EuScript F}}
\def\I{{\rm 1 \hskip-2.9truept l}}

\newcommand{\wh}{\widehat}
\newcommand{\wt}{\widetilde}
\def\N{{\mathbb N}}
\def\N{{\mathbb N}}
\def\R{{\mathbb R}}

\def\H{\mathcal{H}}
\def\K{\mathcal{K}}
\def\cQ{\mathcal{Q}}
\def\L{\mathcal{L}}
\def\D{\mathcal{D}}

\def\S{\mathcal{S}}
\def\E{{\mathbb E}}
\def\P{{\mathbb P}}
\def\Z{{\mathbb Z}}
\def\B{{\cal B}}
\def\M{{\EuScript M}}

\def\os{{\rm Osc}}

\makeatletter \@addtoreset{equation}{section} \makeatother
\newcommand {\qed}%
{%
    {}\hfill
    {}\hfill
    {$\square $}%
    \vspace {0.3cm}%
    \pagebreak [2]%
    \par
}%
\newenvironment{proof}[1]{%
    \vspace{0.3cm}%
    \pagebreak [2]%
    \par%
    \noindent {\bf  Proof~#1\ }}{\qed}%
\newenvironment{remark}{%
    \vspace{0.3cm} \pagebreak [2]%
    \par%
    \refstepcounter{proposition}
    \noindent%
    {\bf Remark~\theproposition\  }}{\qed}%
    
 \newenvironment{remarks}{%
    \vspace{0.3cm} \pagebreak [2]%
    \par%
    \refstepcounter{proposition}
    \noindent%
    {\bf Remarks~\theproposition\  }}{\qed}%
\begin{document}

\title {Lower bound for local oscillations of Hermite processes}
\author{Antoine Ayache \\  
University of Lille, \\
CNRS, UMR 8524 - Laboratoire Paul Painlev\'e,\\
F-59000 Lille, France\\
E-mail: antoine.ayache@univ-lille.fr\\
 }

\maketitle

\begin{abstract} 
The most known example of a class of non-Gaussian stochastic processes which belongs to the homogenous Wiener chaos of an arbitrary order $N>1$ are probably Hermite processes of rank $N$. They generalize fractional Brownian motion (fBm) and Rosenblatt process in a natural way. They were introduced several decades ago. Yet, in contrast with fBm and many other Gaussian and stable stochastic processes and fields 
related to it, few results on path behavior of Hermite processes are available in the literature. For instance the natural issue of whether or not their paths are nowhere differentiable functions has not yet been solved even in the most simple case of the Rosenblatt process. The goal of our article is to derive a quasi-optimal lower bound of the asymptotic behavior of local oscillations of paths of Hermite processes of any rank $N$, which, among other things, shows that these paths are nowhere differentiable functions.   
\noindent 
\end{abstract}

\noindent{\small{\bf Running Title:}\ On the oscillations of Hermite processes.}

\noindent{\small{\bf Key Words.}\  Wiener chaos, self-similar process, Rosenblatt process, path behavior, \\
nowhere differentiability.}


%\newpage
%\clearpage
\section{Introduction and statement of the main result}
\label{sec:intro}
Let an integer $N\ge 1$ and a real number $H\in \big (1-1/(2N),1\big)$ be arbitrary and fixed. The Hermite process of rank $N$ and parameter $H$ is denoted by $X^{H,N}=\{X^{N,H}(t)\}_{t\in\R_+}$. It is a real-valued stochastic process belonging to the homogenous Wiener chaos of order $N$ generated by a real-valued Brownian motion $B=\{B(x)\}_{x\in\R}$ on a probability space $(\O,\F,\P)$, fixed once and for all. More precisely, $\{X^{N,H}(t)\}_{t\in\R_+}$ is defined through the multiple Wiener integral:
\begin{equation}
\label{eq:def-her}
X^{N,H}(t):=\int_{\R^N}' \Big (\int_0^t \prod_{p=1}^N(s-x_p)_+^{H-3/2}\,ds\Big)\, dB(x_1)\ldots dB(x_N),\quad\mbox{ for all $t\in\R_+$\,,}
\end{equation}
with the convention that, for each $(y,\a)\in\R^2$, one has 
\begin{equation}
\label{eq:par-pos}
y_+^\a:=\left\{
\begin{array}{l}
y^\a, \mbox{if $y>0$,}\\
0, \mbox{else.}
\end{array}
\right.
\end{equation}
The symbol $\int_{\R^N}'$ in (\ref{eq:def-her}) denotes integration over $\R^N$ with diagonals $\{x_{p'}=x_{p''}\}$, $p'\ne p''$ excluded. One mentions that two classical books on Wiener chaoses, multiple Wiener integrals and related topics are \cite{janson1997gaussian,nualart}. 

Several decades ago, the well-known papers \cite{Taq75,DoMa79,Taq79} drew important connections between Hermite processes and Non-Central Limit theorem. The process $X^{1,H}$ is the very classical Gaussian fractional Brownian motion (see e.g. \cite{SamTaq,EmMa}). The process $X^{2,H}$ is the non-Gaussian Rosenblatt process which was first introduced in the pioneering article \cite{Ros61}. Since about fifteen years, there has been significantly increasing interest in the study of Hermite processes and more particularly the Rosenblatt process, we refer to the works \cite{Pipiras2004,AB2006,BaTu10,MaTu13,VeTaq13, Arr15,Boj15,BaiTaq14a,BaiTaq14b,BaiTaq15,BaiTaq17b,BaiTaq17a} to cite only a few.

It is well-known that $\{X^{N,H}(t)\}_{t\in\R_+}$ satisfies the following two fundamental properties:
\begin{itemize}
\item[(i)] It is self-similar with exponent $N(H-1)+1\in (1/2,1)$, that is, for each fixed positive real number $a$, the two processes $\big\{X^{N,H}(at)\big\}_{t\in\R_+}$ and $\big\{a^{N(H-1)+1} X^{N,H}(t)\big\}_{t\in\R_+}$ have the same finite-dimensional distributions.
\item[(ii)] It has stationary increments, which means that, for every fixed $t_0\in\R_+$, the two processes $\big\{X^{N,H}(t_0+t)-X^{N,H}(t_0)\big\}_{t\in\R_+}$ and $\big\{X^{N,H}(t)\big\}_{t\in\R_+}$ have the same finite-dimensional distributions.
\end{itemize}
Notice that these two properties imply that 
\begin{equation}
\label{eq:var-X}
\E \Big (\big |X^{N,H}(t')-X^{N,H}(t'')\big |^2\Big)=c_{N,H}|t'-t''|^{2N(H-1)+2}\,,\quad\mbox{for all $(t',t'')\in\R_+^2$,}
\end{equation}
where $c_{N,H}:=\E\big (|X^{N,H}(1)|^2\big)$. Thus using (\ref{eq:var-X}), the inequality $2N(H-1)+2>1$ and the usual Kolmogorov's continuity Theorem, it follows that $\{X^{N,H}(t)\}_{t\in\R_+}$ has a modification with continuous paths. Let us emphasize that throughout our article the process $X^{N,H}=\{X^{N,H}(t)\}_{t\in\R_+}$ is always identified with this modification of it.

Fine study of paths behavior of stochastic processes is a very classical research topic in probability and harmonic analysis whose roots go back to Wiener's works on Brownian paths in the 20's. Plenty of results on this topic have been derived in the Gaussian and stable frameworks. Yet, so far, only few results (see \cite{Mori1986,viens07,arr14}) are known in the framework of non-Gaussian Wiener chaoses. The important article \cite{viens07} provides a general approach allowing to obtain moduli of continuity for a wide class of sub-$n$th chaos processes which includes the Hermite process $X^{N,H}$. An almost sure upper bound of the asymptotic behavior of the local oscillations of the paths of $X^{N,H}$ can be obtained in this way. More precisely, for all $\o\in\O$, for each fixed point $\tau\in (0,+\infty)$ and for any real number $r\in (0,\tau]$, the oscillation on the compact interval $[\tau-r,\tau+r]$ of the path $X^{N,H}(\o)$ is the positive and finite quantity denoted by $\os\big(X^{N,H}(\o),\tau, r\big)$ and defined as:
\begin{equation}
\label{eq:def-os}
\os \big(X^{N,H}(\o),\tau, r\big):=\sup\Big\{\big |X^{N,H}(t',\o)-X^{N,H}(t'',\o)\big|\,:\, (t',t'')\in [\tau-r,\tau+r]^2\Big\}.
\end{equation}
The modulus of continuity for $X^{H,N}(\o)$ obtained thanks to \cite{viens07} allows to derive, for almost all $\o\in\O$, that
\begin{equation}
\label{eq:viens}
\limsup_{r\rightarrow 0} \bigg\{r^{-N(H-1)-1}\,|\log_2 r\, |^{-N/2}\,\sup_{\tau\in I}\os \big(X^{N,H}(\o),\tau, r\big)\bigg\}<+\infty\,,
\end{equation}  
where $I\subset (0,+\infty)$ denotes an arbitrary deterministic compact interval.

It seems natural to look for a non-trivial almost sure lower bound for the asymptotic behavior of  $\os \big(X^{N,H}(\o),\tau, r\big)$, as $r$ goes to $0$. It is important that such a lower bound be valid on an event of probability~$1$ not depending on $\tau$, since by this way it would be possible to derive from it nowhere differentiability of paths of Hermite processes.

Even in the most classical case of the Gaussian Brownian motion such kind of problems are rather difficult to solve. Nowhere differentiability of Brownian paths was first established by Paley, Wiener and Zygmund \cite{pwz33}. Later Dvoretzki \cite{dvo63} was able to obtain sharp lower bounds for their local oscillations. A general and powerful strategy for dealing with this type of problems on everywhere irregularity of paths was first introduced in the early 70's by Berman \cite{ber70,ber72,ber73} in the Gaussian frame and was later extended by Nolan \cite{nol89} to the frame of stable distributions. It relies on a very clever intuitive idea called  the Berman's principle: "the more regular is a local time in the time variable, uniformly in the space variable, the more irregular is the associated stochastic process"; for instance, when the local time is jointly continuous in the time and space variables then the corresponding process has nowhere differentiable paths. Many important developments of this classical and powerful strategy relying on local times are due to Xiao (see e.g. \cite{xia97,xia06,xia13}) who among many other things has introduced strengthened and sectorial versions of the crucial concept of local-nondeterminism. Unfortunately, this strategy can hardly be used in the framework of the Hermite process $X^{N,H}$ since, in contrast with Gaussian and stable processes, there is no explicit and easy exploitable formulas for the characteristic functions of the finite-dimensional distributions of $X^{N,H}$, even in the most simple case of the Rosenblatt process where~$N=2$.

An alternative strategy, relying on wavelet-type series representations, which allows to bound from below the asymptotic behavior of the local oscillations of stochastic processes and fields, was first introduced in \cite{ayache2007}. Several variants of it turned out to be efficient (see e.g. \cite{Ayache2005,ayache2009,arr14}). Yet, in the case of the Hermite process $X^{N,H}$, with $N\ge 3$, no wavelet-type series representation is available so far. Moreover, in the particular case of the Rosenblatt process $X^{2,H}$, the methodology of \cite{ayache2007} seems to be hardly applicable to the wavelet-type series representations of $X^{2,H}$ which were introduced in \cite{Pipiras2004}. Before ending this paragraph, one mentions that the classical strategy due to Berman and the alternative strategy relying on wavelets are presented in detail in the very recent book~\cite{ayache2019}.

Let us now explain in an heuristic way the main ingredient of the new strategy we introduce in the present article. {\em Usually in the literature, fractional Brownian motion, Rosenblatt process and more generally Hermite processes are viewed as stochastic processes whose increments are (rather strongly) correlated, and may even display long range dependence. Our new strategy relies on a different and maybe new way to view increments of these processes in the setting of the study of their path behavior: "many ones of the increments are independent random variables up to negligible remainders".}
 
%Notice that the finiteness of $\os\big(X^{N,H}(\o),\tau, r\big)$ results from the continuity of $X^{N,H}(\o)$.

This new strategy allows us to prove in the present article the following theorem which provides when $\o$ belongs to $\O^*$, a universal event of probability~$1$ not depending on $\tau$, a quasi-optimal lower bound for the asymptotic behavior of  $\os \big(X^{N,H}(\o),\tau, r\big)$, as $r$ goes to $0$.

\begin{theorem}
\label{thm:main}
There exist $\O^*$, a universal event of probability~$1$ not depending on $\tau$, and $c_{N,H}^*$ a (strictly) positive deterministic finite constant only depending $(N,H)$, such that, for all $\o\in\O^*$ and for every $\tau\in (0,+\infty)$, one has
\begin{equation}
\label{thm:main:eq1}
\liminf_{r\rightarrow 0} \bigg\{\Big (r^{-1}\,|\log_2 r\, |\,S\big (|\log_2 r|\big)\Big)^{N(H-1)+1}\,\os \big(X^{N,H}(\o),\tau, r\big)\bigg\}\ge c_{N,H}^*>0\,,
\end{equation}
where $\log_2$ is the logarithm to the base $2$, and $S$ is an arbitrary increasing continuous function on $\R_+$ which satisfies 
\begin{equation}
\label{thm:main:eq2}
S(0)\ge 2\,,\quad\lim_{r\rightarrow +\infty}\frac{z^{\frac{N}{2(1-H)}}}{S(z)}=0\,,\quad \lim_{r\rightarrow +\infty}\frac{z^{\frac{N}{2(1-H)}+\ep}}{S(z)}=+\infty\,,\quad\mbox{for all fixed $\ep>0$,}
\end{equation}
and
\begin{equation}
\label{thm:main:eq2bis}
\sup_{z\in\R_+}\frac{S\big(z+\a \log_2 (2+z)\big)}{S(z)}<+\infty\,, \quad\mbox{for each fixed $\a>0$.}
\end{equation}
\end{theorem}

\begin{remarks}
\label{rem:nondiff} 
\begin{itemize}
\item[(i)] The condition $S(0)\ge 2$ in (\ref{thm:main:eq2}) can be weakened to $S(0)>0$.
\item[(ii)] There are many classes of examples of increasing continuous functions $S$ on $\R_+$ which satisfy (\ref{thm:main:eq2}) and (\ref{thm:main:eq2bis}); a natural one of them is 
\[
S(z):=(2+z)^{\frac{N}{2(1-H)}}\,\big(\log_2(3+z)\big)^\beta\,,\quad\mbox{for all $z\in\R_+$,}
\]
where the positive real number $\beta$ is arbitrary and fixed. 
\item[(iii)] Notice that (\ref{thm:main:eq1}) is quasi-optimal since $r^{-1}$ in it is raised at the same power  $N(H-1)+1$ as in (\ref{eq:viens}).
\end{itemize}
\end{remarks}

An important straightforward consequence of Theorem~\ref{thm:main} is that there is no point in $(0,+\infty)$ at which a typical path of the Hermite process $X^{N,H}$ satisfy a pointwise H\"older condition of order strictly larger than $N(H-1)+1$. More precisely:

\begin{corollary}
\label{cor:main}
Let $\O^*$ be the same event of probability~$1$ as in Theorem~\ref{thm:main}. Let an arbitrary real number $\mu\in \big( N(H-1)+1,1)$. Then, for all $\o\in\O^*$ and for each $\tau\in (0,+\infty)$, one has that
\[
\limsup_{t\rightarrow \tau} \frac{\big | X^{N,H}(t,\o)-X^{N,H}(\tau,\o) \big|}{|t-\tau|^\mu}=+\infty\,.
\]
This clearly implies that, for any $\o\in\O^*$, the path $X^{N,H}(\o)$ is nowhere differentiable on the interval $(0,+\infty)$.
\end{corollary}

\section{Proof of Theorem~\ref{thm:main}}
\label{sec:proof}
%Let $b\in (0,1)$ be arbitrary and fixed. 
First observe that the interval $(0,+\infty)$ can be expressed as the countable union of the open, bounded and overlapping intervals $\big (2^{-1} q, 1+2^{-1} q\big)$, $q\in\Z_+$. So, it is enough to prove the theorem for every $\tau\in\big(2^{-1} q, 1+2^{-1} q\big)$, the nonnegative integer $q$ being arbitrary and fixed. For the sake of simplicity, one assumes that $q=0$, that is $\tau\in (0,1)$; the proof can be done in the same way for any other $q$. Also one assumes that $N\ge 2$, since $N=1$ corresponds to the case of the Gaussian fBm which was studied in detail a long time ago.

For any integers $j\ge 1$ and $k\in\{0,\ldots, 2^{j}-1\}$, one denotes by $\De(j,k)$ the increment of the process $X^{N,H}$ such that 
\begin{equation}
\label{eq:def-delta}
\De(j,k):=X^{N,H}(d_{j,k+1})-X^{N,H}(d_{j,k})\,,
\end{equation} 
where $d_{j,k+1}$ and $d_{j,k}$ are the two dyadic numbers in the interval $[0,1]$ defined as:
\begin{equation}
\label{eq:dya}
d_{j,k+1}:=(k+1)/2^j\quad\mbox{and}\quad d_{j,k}:=k/2^j\,.
\end{equation}
Observe that, in view of (\ref{eq:def-delta}), (\ref{eq:def-her}) and (\ref{eq:par-pos}), the increment $\De (j,k)$ can be expressed as: 
\begin{eqnarray}
\label{eq1:delta}
\De(j,k)&:=&\int_{\R^N}' \Big (\int_{d_{j,k}}^{d_{j,k+1}} \prod_{p=1}^N(s-x_p)_+^{H-3/2}\,ds\Big)\, dB(x_1)\ldots dB(x_N)\\
&=& \int_{\R^N}' \Big (\I_{\In_{j,k}}(x_1,\ldots, x_N)\int_{d_{j,k}}^{d_{j,k+1}} \prod_{p=1}^N(s-x_p)_+^{H-3/2}\,ds\Big)\, dB(x_1)\ldots dB(x_N)\,,\nonumber
\end{eqnarray}
where $\I_{\In_{j,k}}$ is the indicator function of the unbounded rectangle of $\R^N$:
\begin{equation}
\label{eq:def-IN}
\In_{j,k}:= (-\infty, d_{j,k+1}]^N.
\end{equation}
Notice that, in view of the second equality in (\ref{thm:main:eq2}), there exists a positive integer $\jb$ such that, for each integer $j\ge \jb$, one has $S(j)\le 2^{j-4}$. From now on, one always assumes that $j\ge \jb$. One let $e_j$ be the integer part of $S(j)$ (recall that $S(j)\ge 2$), that is 
\begin{equation}
\label{eq:ej}
e_j:=\li S(j)\ri\,,
\end{equation}  
and one denotes by $\L^j$ the non-empty finite set of positive integers defined as: 
\begin{equation}
\label{eq:setL}
\L^j:=\N\cap \big [1, (2^j/e_j)-1\big ]\,.
\end{equation}
Observe that the cardinality of $\L^j$ satisfies, for some positive finite constant $c$ not depending on $j$,
\begin{equation}
\label{eq:card-setL}
\car(\L^j)\le c\,2^j /e_j\,.
\end{equation}
For each $l\in\L^j$, one denotes $\D_{j,l e_j}$ and $\overline{\D}_{j,l e_j}$ the two non-empty subsets of $\In_{j,l e_j}$ (see (\ref{eq:def-IN})) defined as:
\begin{equation}
\label{eq:setD}
\D_{j,l e_j}:=[d_{j, (l-1)e_j+1}, d_{j, l e_j+1}]^N\quad\mbox{and}\quad \overline{\D}_{j,l e_j}:=\In_{j,l e_j}\setminus \D_{j,l e_j}=\{x\in \In_{j,l e_j}: x\notin \D_{j,l e_j}\} \,.
\end{equation}
One clearly has that $\D_{j,l e_j}\cap\overline{\D}_{j,l e_j}=\emptyset$ and $\In_{j,l e_j}=\D_{j,l e_j}\cup\overline{\D}_{j,l e_j}$.
%one lets $\overline{\De}_{j,l}$ be the increment of $X^{N,H}$ such that 
%\begin{equation}
%\label{eq:ode}
%\overline{\De}_{j,l}:=\De(j, l e_j).
%\end{equation} 
%Next 
Thus, using the second equality in (\ref{eq1:delta}), one gets that
\begin{equation}
\label{eq:decom-de}
\De(j, l e_j)=\wt{\De}(j, l e_j)+\breve{\De}(j, l e_j)\,,
\end{equation}
where 
\begin{equation}
\label{eq:tilde-de}
\wt{\De}(j, l e_j)=\int_{\R^N}' \Big (\I_{\D_{j,l e_j}}(x_1,\ldots, x_N)\int_{d_{j,l e_j}}^{d_{j,l e_j+1}} \prod_{p=1}^N(s-x_p)_+^{H-3/2}\,ds\Big)\, dB(x_1)\ldots dB(x_N)
\end{equation}
and 
\begin{equation}
\label{eq:breve-de}
\breve{\De}(j, l e_j)=\int_{\R^N}' \Big (\I_{\overline{\D}_{j,l e_j}}(x_1,\ldots, x_N)\int_{d_{j,l e_j}}^{d_{j,l e_j+1}} \prod_{p=1}^N(s-x_p)_+^{H-3/2}\,ds\Big)\, dB(x_1)\ldots dB(x_N)\,.
\end{equation}

Roughly speaking, as we have already pointed out in the previous section, {\em the main ingredient of our strategy for proving Theorem~\ref{thm:main} consists to show that $\breve{\De}(j, l e_j)$ is in some sense (which will be made more precise in the sequel) negligible with respect to $\wt{\De}(j, l e_j)$, and that the random variables $\wt{\De}(j, l e_j)$, $l\in\L^j$, satisfy the very nice independence property.} To this end, some lemmas are needed.

\begin{lemma}
\label{lem:inde}
For each fixed integer $j\ge \jb$, the random variables $\wt{\De}(j, l e_j)$, $l\in\L^j$, are independent.
\end{lemma}

\begin{proof}{of Lemma~\ref{lem:inde}} The integrand in (\ref{eq:tilde-de}) can be approximated in the sense of the $L^2 (\R^N)$ norm by a sequence $(\phi_n)_{n\in\N}$ of real-valued step functions on $\R^N$ which vanish on the diagonals $\{x_{p'}=x_{p''}\}$, $p'\ne p''$, and also outside of the cube $\D_{j, l e_j}$ (see (\ref{eq:setD})). Thus, using the "isometry" property of multiple Wiener it turns out that 
\begin{equation}
\label{lem:inde:eq1}
\wt{\De}(j, l e_j)=\lim_{n\rightarrow +\infty} \int_{\R^N}' \phi_n (x_1,\ldots, x_N) \, dB(x_1)\ldots dB(x_N),\quad \mbox{(in $L^2 (\O)$).}
\end{equation}
On the other hand, one knows, from the definition of a multiple Wiener integral and from the fact that the step function $\phi_n$ vanishes outside of $\D_{j, l e_j}$, that the random variable $\int_{\R^N}' \phi_n (x_1,\ldots, x_N) \, dB(x_1)\ldots dB(x_N)$ can be expressed as a polynomial function in terms of a finite number of increments $B(a^n_{i+1})-B(a_i^n)$ of the Brownian motion $B$, where $(a_i^n)_i$ is a finite increasing sequence of real numbers belonging to the interval $[d_{j, (l-1)e_j+1}, d_{j, l e_j+1}]$. Thus, it turns out that $\int_{\R^N}' \phi_n (x_1,\ldots, x_N) \, dB(x_1)\ldots dB(x_N)$ is measurable with respect to the $\si$-algebra $\si \big (B(x)-B(y)\,:\, x,y \in (d_{j, (l-1)e_j+1}, d_{j, l e_j+1})\big)$. Combining the latter fact with (\ref{lem:inde:eq1}) it follows that $\wt{\De}(j, l e_j)$ is measurable with respect to the same $\si$-algebra.

Finally, observe that the independence of increments property of the Brownian motion $B$ and the fact the intervals 
$(d_{j, (l-1)e_j+1}, d_{j, l e_j+1})$, $l\in\L^j$, are disjoint imply that the $\si$-algebras $\si \big (B(x)-B(y)\,:\, x,y \in (d_{j, (l-1)e_j+1}, d_{j, l e_j+1})\big)$, $l\in\L^j$, are independent, which in turn entails that the random variables $\wt{\De}(j, l e_j)$, $l\in\L^j$, are independent as well.
\end{proof}

\begin{lemma}
\label{lem:var-tilde}
Let $\wt{c}$ be the finite (strictly) positive constant defined as
\begin{equation}
\label{lem:var-tilde:eq1} 
\wt{c}:=\bigg(\int_{[-1,1]^N}\Big |\int_{0}^{1} \prod_{p=1}^N(u-y_p)_+^{H-3/2}\,ds\Big |^2\, d y_1\ldots d y_N\,\bigg)^{1/2}\,.
\end{equation}
Then, for all integer $j\ge \jb$, one has 
\begin{equation}
\label{lem:var-tilde:eq2}
\inf_{l\in\L^j} \big\| \wt{\De}(j, l e_j)\big\|_{L^2 (\O)}\ge \wt{c}\, 2^{-j(N(H-1)+1)}\,. 
\end{equation}
\end{lemma}

\begin{proof}{of Lemma~\ref{lem:var-tilde}} Using the fact that the integrand in (\ref{eq:tilde-de}) is a symmetric function in the variables $x_1, \ldots , x_N$ and the "isometry" property of multiple Wiener integral one gets, for every $j\ge\jb$ and $l\in\L^j$, that
\[
\big\| \wt{\De}(j, l e_j)\big\|_{L^2 (\O)}^2:=\E \big (|\wt{\De}(j, l e_j)|^2\big )
=N! \int_{\D_{j, l e_j}}\Big |\int_{d_{j,l e_j}}^{d_{j,l e_j+1}} \prod_{p=1}^N(s-x_p)_+^{H-3/2}\,ds\Big |^2\, dx_1\ldots d x_N\,.
\]
Next, using the change of variable $s=d_{j,l e_j}+2^{-j}u$ and (\ref{eq:dya}), one obtains that 
\[
\big\| \wt{\De}(j, l e_j)\big\|_{L^2 (\O)}^2=N! \, 2^{-2j}\int_{\D_{j, l e_j}}\Big |\int_{0}^{1} \prod_{p=1}^N(d_{j,l e_j}+2^{-j}u-x_p)_+^{H-3/2}\,ds\Big |^2\, dx_1\ldots d x_N\,.
\]
Then, the first equality in (\ref{eq:setD}), the change of variables $x_p=d_{j,l e_j}+2^{-j}y_p$, for all $p\in\{1,\ldots, N\}$, the inequality $e_j\ge 2$, for every $j\ge \jb$, and (\ref{lem:var-tilde:eq1}) imply that
\begin{eqnarray*}
\big\| \wt{\De}(j, l e_j)\big\|_{L^2 (\O)}^2 &=&N!\,2^{-(N+2)j}\int_{[1-e_j , 1]^N}\Big |\int_{0}^{1} \prod_{p=1}^N(2^{-j}u-2^{-j}y_p)_+^{H-3/2}\,ds\Big |^2\, d y_1\ldots d y_N\\
&=& N!\,2^{-2j(N(H-1)+1)}\int_{[1-e_j , 1]^N}\Big |\int_{0}^{1} \prod_{p=1}^N(u-y_p)_+^{H-3/2}\,ds\Big |^2\, d y_1\ldots d y_N\\
&\ge & \wt{c}^{\,2}\, 2^{-2j(N(H-1)+1)}\,.
\end{eqnarray*}
\end{proof}

\begin{lemma}
\label{lem:var-breve}
There exists a positive finite constant $c$ such that, for all integer $j\ge \jb$, one has 
\begin{equation}
\label{lem:var-breve:eq1}
\sup_{l\in\L^j} \big\| \breve{\De}(j, l e_j)\big\|_{L^2 (\O)}\le c\, 2^{-j(N(H-1)+1)}\,e_j^{H-1}\,. 
\end{equation}
\end{lemma}

\begin{proof}{of Lemma~\ref{lem:var-breve}} Using the fact that the integrand in (\ref{eq:breve-de}) is a symmetric function in the variables $x_1, \ldots , x_N$, the "isometry" property of multiple Wiener integral, (\ref{eq:def-IN}), (\ref{eq:setD}), the inequality $d_{j,(l-1)e_j+1}<d_{j,l e_j}$ and the fact that $z\mapsto z^{H-3/2}$ is a decreasing function on $(0,+\infty)$, one gets that
\begin{eqnarray}
\label{lem:var-breve:eq2}
&&\big\| \breve{\De}(j, l e_j)\big\|_{L^2 (\O)}^2:=\E \big (|\breve{\De}(j, l e_j)|^2\big )\nonumber\\
&& =N! \int_{\overline{\D}_{j, l e_j}}\Big |\int_{d_{j,l e_j}}^{d_{j,l e_j+1}} \prod_{p=1}^N(s-x_p)_+^{H-3/2}\,ds\Big |^2\, dx_1\ldots d x_N\nonumber\\
&& \le N\cdot N! \int_{-\infty}^{d_{j,(l-1)e_j+1}}\bigg (\nonumber\\
&& \hspace{1cm}\int_{\R^{N-1}}\Big |\int_{d_{j,l e_j}}^{d_{j,l e_j+1}} (s-x_N)^{H-3/2}\prod_{p=1}^{N-1}(s-x_p)_+^{H-3/2}\,ds\Big |^2\, dx_1\ldots d x_{N-1}\bigg)\, dx_N \nonumber\\
&& \le N\cdot N! \int_{-\infty}^{d_{j,(l-1)e_j+1}}(d_{j,l e_j}-x_N )^{2H-3}\, dx_N\\
&&\hspace{2cm}\times\int_{\R^{N-1}}\Big |\int_{d_{j,l e_j}}^{d_{j,l e_j+1}} \prod_{p=1}^{N-1}(s-x_p)_+^{H-3/2}\,ds\Big |^2\, dx_1\ldots d x_{N-1}\,.\nonumber
\end{eqnarray}
It results from standard computations and (\ref{eq:dya}) that
\begin{eqnarray}
\label{lem:var-breve:eq3} 
&& \int_{-\infty}^{d_{j,(l-1)e_j+1}}(d_{j,l e_j}-x_N )^{2H-3}\, dx_N=\frac{(d_{j,l e_j}-d_{j,(l-1)e_j+1})^{2H-2}}{2-2H}\nonumber\\
&&=\frac{(d_{j,e_j-1})^{2H-2}}{2-2H}= \frac{(e_j-1)^{2H-2}\, 2^{2j (1-H)}}{2-2H} \le \frac{e_j^{2H-2}\, 2^{-2(j+1) (H-1)}}{2-2H}\,.
\end{eqnarray}
Moreover, the "isometry" property of multiple Wiener integral, (\ref{eq:def-her}) and (\ref{eq:var-X}) imply that 
\begin{eqnarray}
\label{lem:var-breve:eq4} 
&& (N-1)! \int_{\R^{N-1}}\Big |\int_{d_{j,l e_j}}^{d_{j,l e_j+1}} \prod_{p=1}^{N-1}(s-x_p)_+^{H-3/2}\,ds\Big |^2\, dx_1\ldots d x_{N-1}\\
&&=\E \Big (\big |X^{N-1,H}(d_{j,l e_j+1})-X^{N-1,H}(d_{j,l e_j})\big |^2\Big)=c_{N-1,H}\, 2^{-2j((N-1)(H-1)+1)}\,. \nonumber
\end{eqnarray}
Then combining (\ref{lem:var-breve:eq2}), (\ref{lem:var-breve:eq3}) and (\ref{lem:var-breve:eq4}), one obtains (\ref{lem:var-breve:eq1}).
\end{proof}

\begin{lemma}
\label{lem:maj-breve}
There are $\breve{\O}$ an event of probability~$1$ and $\breve{C}$ a positive finite random variable such that on $\breve{\O}$, for every integer $j\ge \jb$, one has
\begin{equation}
\label{lem:maj-breve:eq1}
\sup_{l\in\L^j} \big | \breve{\De}(j, l e_j)\big |\le \breve{C}\, 2^{-j(N(H-1)+1)}\,e_j^{H-1}\, j^{N/2}\,.
\end{equation}
\end{lemma}

In order to show that Lemma~\ref{lem:maj-breve} holds one needs the following lemma which is in fact Theorem~6.7 on page 82 of the well-known book~\cite{janson1997gaussian}.
\begin{lemma}
\label{lem:Jan}
For any fixed integer $N\ge 1$, there exists a (strictly) positive finite universal deterministic constant $\breve{c}_N$ such that, for every random variable $\chi$ belonging to the Wiener chaos of order $N$ and for each real number $y\ge 2$, one has
\begin{equation}
\label{lem:Jan:eq1}
\P\Big (|\chi|>y\|\chi\|_{L^2(\O)}\Big)\le \exp\Big (-\breve{c}_N\,y^{2/N}\Big)\, ,
\end{equation}
where $\|\chi\|_{L^2(\O)}:=\Big (\E\big [|\chi|^2\big]\Big)^{1/2}$.
\end{lemma}

\begin{proof}{of Lemma~\ref{lem:maj-breve}}  Let $\kappa\geq 2$ be a finite deterministic constant which will be soon defined more precisely. For all integer $j\ge\jb$, one has
\begin{eqnarray}
\label{lem:maj-breve:eq2}
&& \P\bigg(\sup_{l\in\L^j} \big | \breve{\De}(j, l e_j)\big |>\kappa\,j^{N/2}\,\sup_{l\in\L^j} \big \| \breve{\De}(j, l e_j)\big \|_{L^2(\O)}\bigg) \nonumber\\
&& \le \sum_{l\in\L^j}\P\bigg(\big | \breve{\De}(j, l e_j)\big | >\kappa\,j^{N/2}\,\sup_{l\in\L^j} \big \| \breve{\De}(j, l e_j)\big \|_{L^2(\O)}\bigg)\nonumber\\
&&\le \sum_{l\in\L^j}\P\bigg(\big | \breve{\De}(j, l e_j)\big | >\kappa\,j^{N/2}\, \big \| \breve{\De}(j, l e_j)\big \|_{L^2(\O)}\bigg) \, .
\end{eqnarray}
Moreover, one can make use of Lemma~\ref{lem:Jan} in order to bound from above the probabilities in (\ref{lem:maj-breve:eq2}). By this way, for all $l\in\L^j$, one gets
\begin{equation}
\label{lem:maj-breve:eq3}
\P\bigg(\big | \breve{\De}(j, l e_j)\big | >\kappa\,j^{N/2}\, \big \| \breve{\De}(j, l e_j)\big \|_{L^2(\O)}\bigg) \le \exp\big (-\breve{c}_N\,\kappa\,j\big),
\end{equation}
where $\breve{c}_N$ is the same deterministic positive finite constant as in (\ref{lem:Jan:eq1}). Next combining (\ref{lem:maj-breve:eq2}) and (\ref{lem:maj-breve:eq3}) with (\ref{eq:card-setL}) and the inequality $e_j\ge 2$, one obtains that
\begin{eqnarray}
\label{lem:maj-breve:eq4}
&& \P\bigg(\sup_{l\in\L^j} \big | \breve{\De}(j, l e_j)\big |>\kappa\,j^{N/2}\,\sup_{l\in\L^j} \big \| \breve{\De}(j, l e_j)\big \|_{L^2(\O)}\bigg)\nonumber\\
&& \le c\, 2^j e_j^{-1}\,\exp\big (-\breve{c}_N\,\kappa\,j\big)=c\,\exp\big (-(\breve{c}_N\,\kappa-\log 2)j\big),
\end{eqnarray}
where $c$ denotes the same constant as in (\ref{eq:card-setL}). One can assume that the finite constant $\kappa$ is chosen such that $\kappa>(\log 2)/\breve{c}_N$. Then, one can derive from (\ref{lem:maj-breve:eq4}) that
\[
\sum_{j=\jb}^{+\infty} \P\bigg(\sup_{l\in\L^j} \big | \breve{\De}(j, l e_j)\big |>\kappa\,j^{N/2}\,\sup_{l\in\L^j} \big \| \breve{\De}(j, l e_j)\big \|_{L^2(\O)}\bigg)< +\infty.
\]
Thus, the Borel-Cantelli's Lemma implies that there exist $\breve{\Omega}$ an event of probability~$1$ and $\breve{C}_0$ a positive finite random variable,
such that, for all $\o\in\breve{\Omega}$ and for every integer $j\ge\jb$, one has
\begin{equation}
\label{lem:maj-breve:eq5}
\sup_{l\in\L^j} \big | \breve{\De}(j, l e_j,\o)\big|\le \breve{C}_0(\o) \,j^{N/2}\,\sup_{l\in\L^j} \big \| \breve{\De}(j, l e_j)\big \|_{L^2(\O)}
\end{equation}
Finally, (\ref{lem:maj-breve:eq1}) results from (\ref{lem:var-breve:eq1}) and (\ref{lem:maj-breve:eq5}).
\end{proof} 

The following lemma is a straightforward consequence of Theorem 6.9 and Remark~6.10 on page 82 of the well-known book~\cite{janson1997gaussian}.

\begin{lemma}
\label{lem2:Jan}
For any fixed integer $N\ge 1$, there exists a universal deterministic constant $\ga_N$ satisfying
\begin{equation}
\label{lem2:Jan:eq0}
0\le \ga_N <1\,,
\end{equation}
such that, for each random variable $\chi$ belonging to the Wiener chaos of order $N$ one has
\begin{equation}
\label{lem2:Jan:eq1}
\P\Big (|\chi| < 2^{-1}\|\chi\|_{L^2(\O)}\Big)\le \ga_N\, ,
\end{equation}
where $\|\chi\|_{L^2(\O)}:=\Big (\E\big [|\chi|^2\big]\Big)^{1/2}$.
\end{lemma}

\begin{remark}
\label{rem:parti}
Observe that, in view of the fact that the nonnegative constant $\ga_N$ is {\em strictly} smaller than~$1$ (see  (\ref{lem2:Jan:eq0})), there exists an integer $n_0\ge 2$, only depending on $N$, such that
\begin{equation}
\label{rem:parti:eq1}
0\le 2\ga_N^{n_0}<1\,.
\end{equation}
Also observe that the second equality in (\ref{thm:main:eq2}) and (\ref{eq:ej}) imply that there is a positive integer $\jbu\ge\jb$ such that, for all integer $j\ge\jbu$ one has
\begin{equation}
\label{rem:parti:eq2} 
\diam \big (\big [1, (2^j/e_j)-1\big ]\big):=(2^j/e_j)-1-1=(2^j/e_j)-2\ge 10\,n_0 j\,.
\end{equation}
For each fixed integer $j\ge \jbu$, the integer $M_j\ge 10$ denotes the integer part of $(n_0 j)^{-1}\big ((2^j/e_j)-2\big)$, that is 
\begin{equation}
\label{rem:parti:eq3} 
M_j:= \li (n_0 j)^{-1}\big ((2^j/e_j)-2\big) \ri\,.
\end{equation}
Also, for each fixed integer $j\ge \jbu$, one denotes by $(U_m ^j)_{m\in\{0,1,\ldots, M_j\}}$ subdivision of the interval $\big [1, (2^j/e_j)-1\big ]$ by the $M_j+1$ points defined as:
\begin{equation}
\label{rem:parti:eq4} 
U_{M_j}^j:=(2^j/e_j)-1\quad\mbox{and} \quad U_m ^j:=1+m(n_0 j)\,,\,\mbox{for all $m\in\{0,1,\ldots, M_j-1\}$;}
\end{equation}
notice that 
\begin{equation}
\label{rem:parti:eq5}
n_0 j\le U_{M_j}^j  - U_{M_j-1}^j< 2(n_0 j)\,.
\end{equation}
Moreover, for every fixed integer $j\ge \jbu$, one lets $(\wt{\La}_m ^j)_{m\in\{1,\ldots, M_j\}}$ be the sequence of the nonnegative finite random variables defined,  for all $m\in \{1,\ldots, M_j\}$, as:
\begin{equation}
\label{rem:parti:eq6}
\wt{\La}_m ^j:=\sup\Big\{\big|\wt{\De}(j, l e_j)\big|\,:\, l\in\L_m ^j \Big\}\,,
\end{equation}
where $\wt{\De}(j, l e_j)$ is as in (\ref{eq:tilde-de}) and 
\begin{equation}
\label{rem:parti:eq7}
\L_m ^j:=\N\cap \big [U_{m-1}^j, U_m^j\big]\,.
\end{equation}
Observe that, in view of (\ref{eq:setL}), (\ref{rem:parti:eq4}), (\ref{rem:parti:eq5}) and  (\ref{rem:parti:eq7}), one has
 \begin{equation}
\label{rem:parti:eq8}
\L^j=\bigcup_{m=1}^{M_j}\L_m ^j
\end{equation}
and
\begin{equation}
\label{rem:parti:eq9}
\car (\L_m ^j) > n_0 j\,, \quad\mbox{for all $m\in\{1,\ldots, M_j\}$.}
\end{equation}
\end{remark}

\begin{lemma}
\label{lem:maj-tilde}
One denotes by $\wt{c}$ the same (strictly) positive deterministic constant as in (\ref{lem:var-tilde:eq1}). Then, there is $\wt{\O}$ an event of probability~$1$ such that on $\wt{\O}$ one has
\begin{equation}
\label{lem:maj-tilde:eq1}
\liminf_{j\rightarrow +\infty} \Big\{ 2^{j(N(H-1)+1)} \inf_{1\le m \le M_j}\wt{\La}_m ^j\Big\}\ge 2^{-1}\,\wt{c}>0\,. 
\end{equation}
\end{lemma}

\begin{proof}{of Lemma~\ref{lem:maj-tilde}} Let $\wt{\O}$ be the event defined as:
\begin{equation}
\label{lem:maj-tilde:eq2}
\wt{\O}:=\bigcup_{J=\jbu}^{+\infty}\bigcap_{j=J}^{+\infty} \bigcap_{m=1}^{M_j} \big\{\o\in\O\,:\, \wt{\La}_m ^j (\o)\ge 2^{-1}\,\wt{c}\,2^{-j(N(H-1)+1)}\big\}\,,
\end{equation}
where the fixed positive integer $\jbu$ is as in Remark~\ref{rem:parti}. In order to show that the lemma holds, it is enough to prove $\P(\wt{\O})=1$ which is equivalent to prove that
\begin{equation}
\label{lem:maj-tilde:eq3}
\P(\O\setminus\wt{\O})=0\,.
\end{equation}
Notice that, in view of (\ref{lem:maj-tilde:eq2}) the event $\O\setminus\wt{\O}$ can be expressed as:
\[
\O\setminus\wt{\O}:=\bigcap_{J=\jbu}^{+\infty}\bigcup_{j=J}^{+\infty} \bigcup_{m=1}^{M_j} \big\{\o\in\O\,:\, \wt{\La}_m ^j (\o) < 2^{-1}\,\wt{c}\,2^{-j(N(H-1)+1)}\big\}\,.
\]
Thus, one knows from the Borel-Cantelli's Lemma that in order to derive (\ref{lem:maj-tilde:eq3}) it is enough to prove that
\begin{equation}
\label{lem:maj-tilde:eq4}
\sum_{j=\jbu}^{+\infty}\P\Big (\bigcup_{m=1}^{M_j} \big\{\o\in\O\,:\, \wt{\La}_m ^j (\o) < 2^{-1}\,\wt{c}\,2^{-j(N(H-1)+1)}\big\}\Big)<+\infty\,.
\end{equation}
Putting together (\ref{rem:parti:eq6}), Lemma~\ref{lem:inde}, (\ref{rem:parti:eq8}), Lemma~\ref{lem:var-tilde}, Lemma~\ref{lem2:Jan}, (\ref{rem:parti:eq9}), (\ref{rem:parti:eq3}) and the inequality $e_j\ge 2$, one gets, for every $j\ge\jbu$,
\begin{eqnarray}
\label{lem:maj-tilde:eq5}
&& \P\Big (\bigcup_{m=1}^{M_j} \big\{\o\in\O\,:\, \wt{\La}_m ^j (\o) < 2^{-1}\,\wt{c}\,2^{-j(N(H-1)+1)}\big\}\Big)\ \le \sum_{m=1}^{M_j} \P\Big (\wt{\La}_m ^j <2^{-1}\,\wt{c}\,2^{-j(N(H-1)+1)}\Big)\nonumber\\
&& =\sum_{m=1}^{M_j}\P\bigg (\bigcap_{l\in\L_m^j}\Big\{\o\in\O\,:\, \big | \wt{\De}(j, l e_j,\o)\big |<2^{-1}\,\wt{c}\,2^{-j(N(H-1)+1)}\Big\}\bigg)\nonumber\\
&&=\sum_{m=1}^{M_j}\,\prod_{l\in\L_m^j}\P\Big (\big | \wt{\De}(j, l e_j)\big |<2^{-1}\,\wt{c}\,2^{-j(N(H-1)+1)}\Big) < M_j \,\ga_N ^{n_0 j} < \big (2\ga^{n_0}\big)^j\,.
\end{eqnarray}
Thus, (\ref{lem:maj-tilde:eq4}) follows from (\ref{rem:parti:eq1}) and (\ref{lem:maj-tilde:eq5}). 
\end{proof}

\begin{lemma}
\label{lem:min-de}
Using some of the notations previously introduced in Remark~\ref{rem:parti}, for every fixed integer $j\ge \jbu$, one lets $(\La_m ^j)_{m\in\{1,\ldots, M_j\}}$ be the sequence of the nonnegative finite random variables defined, for all $m\in \{1,\ldots, M_j\}$, as:
\begin{equation}
\label{lem:min-de:eq1}
\La_m ^j:=\sup\Big\{\big|\De(j, l e_j)\big|\,:\, l\in\L_m ^j \Big\}\,,
\end{equation}
where $\De(j, l e_j)$ is as in (\ref{eq:decom-de}). Moreover, one denotes by $\O^*$ the event of probability~$1$ defined as $\O^*:=\breve{\O}\cap\widetilde{\O}$ (see Lemmas~\ref{lem:maj-breve} and \ref{lem:maj-tilde}).
Then, one has on $\O^*$
\begin{equation}
\label{lem:min-de:eq2}
\liminf_{j\rightarrow +\infty} \Big\{ 2^{j(N(H-1)+1)} \inf_{1\le m \le M_j} \La_m ^j\Big\}\ge 2^{-1}\,\wt{c}>0\,. 
\end{equation}
where $\wt{c}$ is the same strictly positive deterministic constant as in (\ref{lem:var-tilde:eq1}).
\end{lemma}

\begin{proof}{of Lemma~\ref{lem:min-de}} Using (\ref{eq:decom-de}), (\ref{lem:min-de:eq1}), the triangle inequality, (\ref{rem:parti:eq6}), (\ref{rem:parti:eq8}) and Lemma~\ref{lem:maj-breve}, one obtains on the event $\O^*$, for every $j\ge\jbu$,
\begin{eqnarray}
\label{lem:min-de:eq3}
&&2^{j(N(H-1)+1)}\inf_{1\le m \le M_j} \La_m ^j\ge 2^{j(N(H-1)+1)} \inf_{1\le m \le M_j}\Big\{\wt{\La}_m ^j-\sup_{l\in\L_m^j}\big|\breve{\De}(j, l e_j)\big|\Big\}\nonumber\\
&& \ge 2^{j(N(H-1)+1)} \bigg ( \inf_{1\le m \le M_j}\wt{\La}_m ^j-\sup_{l\in\L^j}\big|\breve{\De}(j, l e_j)\big|\bigg)\nonumber\\
&& \ge  \bigg ( 2^{j(N(H-1)+1)} \inf_{1\le m \le M_j}\wt{\La}_m ^j\bigg)-\breve{C}\,e_j^{H-1}\, j^{N/2}.
\end{eqnarray}
Moreover, it follows (\ref{eq:ej}) and the first equality in (\ref{thm:main:eq2}) that
\begin{equation}
\label{lem:min-de:eq4}
\lim_{j\rightarrow +\infty} e_j^{H-1}\, j^{N/2}=0\,.
\end{equation}
Then combining (\ref{lem:min-de:eq3}) and (\ref{lem:min-de:eq4}) with (\ref{lem:maj-tilde:eq1}), one gets (\ref{lem:min-de:eq2}).
\end{proof}

We are now in position to complete the prove Theorem~\ref{thm:main}.\\
{\bf End of the proof of Theorem~\ref{thm:main}} In all the sequel the point $\tau\in (0,1)$ is arbitrary and fixed and, for integer $j\ge \jbu$, one sets
\begin{equation}
\label{thm:main:eq3}
l_j (\tau):=\li 2^j\tau/e_j\ri\,,
\end{equation}
where, as usual, $\li\cdot \ri$ denotes the integer part function. One clearly has that 
\begin{equation}
\label{thm:main:eq4}
0\le \tau-\frac{l_j (\tau)e_j}{2^j}<\frac{e_j}{2^j}\,.
\end{equation}
Moreover, it easily follows from (\ref{thm:main:eq3}), the inequalities $0<\tau<1$, (\ref{eq:setL}), (\ref{eq:ej}) and the second equality in (\ref{thm:main:eq2}) that there exists a positive integer $\jbut\ge\jbu$, such that for all integer $j\ge\jbut$, one has 
\begin{equation}
\label{thm:main:eq5}
l_j (\tau)\in\L^j
\end{equation}
and
\[
\Big [\tau-\frac{12 n_0 j S(j+1)}{2^j}, \tau+\frac{12 n_0 j S(j+1)}{2^j}\Big]\subset (0,1)\,,
\]
where the fixed integer $n_0\ge 2$ is as in (\ref{rem:parti:eq1}). In view of (\ref{thm:main:eq5}) and (\ref{rem:parti:eq8}), there is $m_j (\tau)\in\{1,\ldots, M_j\}$ such that 
\begin{equation}
\label{thm:main:eq6}
l_j (\tau)\in\L_{m_j (\tau)}^j\,.
\end{equation}
Thus, one can derive from (\ref{rem:parti:eq7}), (\ref{rem:parti:eq5}) and (\ref{rem:parti:eq4}) that 
\begin{equation}
\label{thm:main:eq7}
\big | l_j (\tau)-l\big|< 2n_0 j\,,\quad\mbox{for all $l\in\L_{m_j (\tau)}^j$.}
\end{equation}
Next, combining (\ref{thm:main:eq7}), (\ref{thm:main:eq4}), (\ref{eq:ej}) and the inequality $n_0\ge 2$, one gets that 
\begin{equation}
\label{thm:main:eq8}
 \Big |\tau-\frac{l e_j}{2^j}\Big| < \frac{(2n_0 j+1) e_j}{2^j}\le \frac{3n_0 j e_j}{2^j}\le\frac{3n_0 j S(j)}{2^j}\,,\quad\mbox{for all $l\in\L_{m_j (\tau)}^j$,}
 \end{equation}
 and
 \begin{equation}
\label{thm:main:eq9}
 \quad \Big |\tau-\frac{l e_j+1}{2^j}\Big|< \frac{(2n_0 j+1) e_j+1}{2^j}\le \frac{3n_0 j e_j}{2^j}\le\frac{3n_0 j S(j)}{2^j}\,,
 \quad\mbox{for all $l\in\L_{m_j (\tau)}^j$.}
  \end{equation}
Next it follows from (\ref{lem:min-de:eq1}), (\ref{eq:def-delta}), (\ref{eq:dya}), (\ref{eq:def-os}), (\ref{thm:main:eq8}) and (\ref{thm:main:eq9}) that, for all $\o\in\O$,
 \begin{equation}
\label{thm:main:eq10}
\os\Big (X^{N,H}(\o)\,,\tau\,, \frac{3n_0 j S(j)}{2^j}\Big)\ge \La_{m_j (\tau)}^j (\o)\ge \inf_{1\le m \le M_j} \La_m ^j(\o)\,.
\end{equation}
Thus, assuming that $\eta\in \big (0, 2^{-1}\wt{c}\big)$ is arbitrary and fixed, one can derive from (\ref{thm:main:eq10}) and Lemma~\ref{lem:min-de} that, for all $\o\in\O^*$, there exists an integer $j_3=j_3 (\tau,\eta,\o)\ge \jbut$ such that, for all integer $j\ge j_3$, one has 
 \begin{equation}
\label{thm:main:eq11}
2^{j(N(H-1)+1)}\os\Big (X^{N,H}(\o)\,,\tau\,, \frac{3n_0 j S(j)}{2^j}\Big)\ge \eta\,.
\end{equation}
Next, let $\rho$ be an arbitrary positive real number such $\rho\le 2^{-j_3}$, one sets 
\begin{equation}
\label{thm:main:eq12}
j^*(\rho):=\li-\log_2 \rho\ri\,.
\end{equation}
One clearly has that $j^*(\rho)\ge j_3$ and 
\begin{equation}
\label{thm:main:eq13}
2^{j^*(\rho)}\le \rho^{-1}<2^{j^*(\rho)+1}\,.
\end{equation}
Thus, (\ref{thm:main:eq11}), (\ref{thm:main:eq12}), (\ref{thm:main:eq13}), (\ref{eq:def-os}) and the fact that $S$ is an increasing function imply that
\begin{eqnarray}
\label{thm:main:eq14}
&& \rho^{-N(H-1)-1}\,\os\Big (X^{N,H}(\o)\,,\tau\,, 6n_0(-\log_2 \rho)S(-\log_2 \rho) \rho\Big)\nonumber\\
&& \ge 2^{j^* (\rho)(N(H-1)+1)}\os\Big (X^{N,H}(\o)\,,\tau\,, \frac{3n_0 j^*(\rho) S(j^*(\rho))}{2^{j^*(\rho)}}\Big)\ge \eta\,.
\end{eqnarray}
Next, let $r$ be an arbitrary positive real number such that $r\le 2^{-j_3}$. Using the latter inequality, the inequality $n_0\ge 2$, and the fact the function $S$ is with values in $[2,+\infty)$ (see the inequality in (\ref{thm:main:eq2})), one obtains that
$\nu(r)\le r\le 2^{-j_3}$, where
\begin{equation}
\label{thm:main:eq15}
\nu(r):=\frac{r}{6n_0(-\log_2 r)S(-\log_2 r)}\,.
\end{equation}
Thus, one can take in (\ref{thm:main:eq14}) $\rho=\nu (r)$, and by this way one obtains, for all $r\in (0, 2^{-j_3}]\subseteq (0,1/2]$, that
\begin{equation}
\label{thm:main:eq15bis}
\big(\nu(r)\big)^{-N(H-1)-1}\,\os\Bigg (X^{N,H}(\o)\,,\tau\,, \bigg(\frac{\big(-\log_2 \nu(r) \big)S\big(-\log_2 \nu(r)\big)}{(-\log_2 r)S(-\log_2 r)} \bigg)r\Bigg)\ge \eta\,.
\end{equation}
Let us now show that 
\begin{equation}
\label{thm:main:eq16}
\th_0:=\sup_{r\in (0,1/2]} \frac{\big(-\log_2 \nu(r) \big)S\big(-\log_2 \nu(r)\big)}{(-\log_2 r)S(-\log_2 r)}<+\infty\,.
\end{equation}
Let us set $z=z(r):=-\log_2 r$, where $r\in (0,1/2]$ is arbitrary. Using (\ref{thm:main:eq15}), the inequality $z\ge 1$, the second equality in (\ref{thm:main:eq2}), and the fact that $S$ is an increasing function, one gets that
\begin{equation}
\label{thm:main:eq17}
S\big(-\log_2 \nu(r)\big)=S\big (z+\log_2 (6n_0)+\log_2 (z)+\log_2 (S(z))\big)\le S\big (z+\a_0\log_2 (2+z)\big)\,,
\end{equation}
where $\a_0$ is a positive finite constant only depending on $N$ and $S$. Thus, it follows from (\ref{thm:main:eq17}) and (\ref{thm:main:eq2bis}) that 
\begin{equation}
\label{thm:main:eq18}
\sup_{r\in (0,1/2]} \frac{S\big(-\log_2 \nu(r)\big)}{S(-\log_2 r)}\le \sup_{z\in [1,+\infty)}\frac{S\big(z+\a_0 \log_2 (2+z)\big)}{S(z)}<+\infty\,.
\end{equation}
Similarly to (\ref{thm:main:eq17}) it can be shown that $\big(-\log_2 \nu(r)\big)\le z+\a_0\log_2 (2+z)$, for all $r\in (0,1/2]$; thus, one gets that
\begin{equation}
\label{thm:main:eq19}
\sup_{r\in (0,1/2]} \frac{\big(-\log_2 \nu(r)\big)}{(-\log_2 r)}\le \sup_{z\in [1,+\infty)}\frac{z+\a_0 \log_2 (2+z)}{z}<+\infty\,.
\end{equation}
Then (\ref{thm:main:eq16}) results from (\ref{thm:main:eq18}) and (\ref{thm:main:eq19}). Next, combining 
(\ref{thm:main:eq15bis}) and (\ref{thm:main:eq16}) with (\ref{eq:def-os}) and the fact that $\eta\in \big (0, 2^{-1}\,\wt{c}\big)$ is arbitrary, it follows that, for all $\tau\in (0,1)$ and $\o\in\O^*$,
\[
\liminf_{r\rightarrow 0}\bigg\{\big(\nu(r)\big)^{-N(H-1)-1}\,\os\Big(X^{N,H}(\o)\,,\tau\,, \th_0 r\Big)\bigg\}\ge
2^{-1}\,\wt{c} 
\]
and consequently that
\begin{equation}
\label{thm:main:eq20}
\liminf_{r\rightarrow 0}\bigg\{\Big(\nu\big(\th_0^{-1}\,r\big)\Big)^{-N(H-1)-1}\,\os\Big(X^{N,H}(\o)\,,\tau\,,  r\Big)\bigg\}\ge 2^{-1}\,\wt{c}\,.
\end{equation}
Finally, in view of (\ref{thm:main:eq15}) and (\ref{thm:main:eq20}), in order to derive (\ref{thm:main:eq1}) it is enough to show that 
\begin{equation}
\label{thm:main:eq21}
\limsup_{r\rightarrow 0} \frac{\nu(r)}{\nu\big(\th_0^{-1}\,r\big)}<+\infty\,.
\end{equation}
One knows from (\ref{thm:main:eq15}) that, for each $r>0$ small enough, one has 
\begin{eqnarray*}
\frac{\nu(r)}{\nu\big(\th_0^{-1}\,r\big)}&=&\frac{\th_0\big(-\log_2 (\th_0^{-1}\, r )\big)S\big (-\log_2 (\th_0^{-1}\, r )\big)}{(-\log_2 r)S(-\log_2 r)}\\
&\le&\th_0\big (1+\log_2(\th_0)\big)\times\frac{S\big (-\log_2 (r)+\log_2(\th_0)\big)}{S(-\log_2 r)}\,.
\end{eqnarray*}
Thus, using (\ref{thm:main:eq2bis}) one obtains (\ref{thm:main:eq21}).
\qed

\section*{Acknowledgements}
This work has been partially supported by the Labex CEMPI (ANR-11-LABX-0007-01) and the GDR 3475 (Analyse Multifractale). 

\bibliographystyle{plain}
\bibliography{biblio-Herm}

\addcontentsline{toc}{section}{Réfèrences}
\end{document}